\documentclass[12pt,a4paper]{article}
\usepackage{amssymb,amsfonts,amsthm,amsmath}
\usepackage{graphicx}

\def\hang{\hangindent\parindent}
\def\rf{\par\noindent\hang}
\def\ssk{\smallskip}

\def\rf{\par\noindent\hang}
\def\nin{\noindent}

\newtheorem{theorem}{Theorem}

\newtheorem{lemma}{Lemma}
\newtheorem{corollary}{Corollary}

\begin{document}

\baselineskip=20pt

\begin{center} \large{{\bf ADMISSIBILITY OF THE USUAL CONFIDENCE INTERVAL IN
LINEAR REGRESSION}}
\end{center}

\smallskip

\begin{center} SHORT RUNNING TITLE: ADMISSIBILITY IN REGRESSION
\end{center}

\smallskip


\begin{center}
\large{PAUL KABAILA$^{1*}$, KHAGESWOR GIRI$^2$ {\normalsize AND} HANNES LEEB$^3$}
\end{center}

\begin{center}
{\sl 1. Department of Mathematics and Statistics \\
La Trobe University \\
Bundoora Victoria 3086 \\
Australia }
\end{center}

\begin{center}
{\sl 2. Future Farming Systems Research \\
Department of Primary Industries \\
600 Sneydes Road \\ Werribee 3030 Victoria \\
Australia }
\end{center}

\begin{center}
{\sl 3. Department of Statistics \\
University of Vienna \\
Universit\"atsstr. 5/3,
A-1010 Vienna  \\
Austria}
\end{center}

\date{}
\pagenumbering{arabic}


\begin{center}
 {\bf Summary}
\end{center}

\ssk

\noindent Consider a linear regression model with independent and identically normally distributed random
errors. Suppose that the parameter of interest is a specified linear combination of the regression
parameters. We prove that the usual confidence interval for this parameter is admissible within
a broad class of confidence intervals.

\noindent

\bigskip
\bigskip

\nin{\sl Keywords:} admissibility; compromise decision theory;
confidence interval; decision theory.

\bigskip
\bigskip

\nin $^*$ Author to whom correspondence should be addressed.
Department of Mathematics and Statistics, La Trobe University,
Victoria 3086, Australia.
Tel.: +61 3 9479 2594, fax:
 +61 3 9479 2466,
 {e-mail:} P.Kabaila@latrobe.edu.au

\vfil
\eject

\begin{center}
{\large{\bf 1. Introduction}}
\end{center}

\medskip

 Consider the linear regression model
 $Y = X \beta + \varepsilon$,
 where $Y$ is a random $n$-vector of responses, $X$ is a known $n \times p$ matrix with linearly
independent columns, $\beta$ is an unknown parameter $p$-vector and
$\varepsilon \sim N(0, \sigma^2 I_n)$ where $\sigma^2$ is an unknown positive parameter.
Let $\hat \beta$ denote the least squares estimator of $\beta$.
Also, define $\hat \sigma^2 = (Y - X \hat \beta)^T (Y - X \hat \beta)/(n-p)$.

Suppose that the parameter of interest is $\theta = a^T \beta$ where $a$ is a given
$p$-vector ($a \ne 0$). We seek a $1-\alpha$ confidence interval for $\theta$.
Define the quantile $t(m)$ by the requirement that
$P \big(-t(m) \le T \le t(m) \big) = 1-\alpha$ for $T \sim t_m$.
Let $\hat \Theta$ denote $a^T \hat \beta$, i.e. the least squares estimator of
$\theta$. Also
let $v_{11}$ denote the variance of $\hat \Theta$ divided by $\sigma^2$.
The usual $1-\alpha$ confidence interval for $\theta$ is
\begin{equation*}
I = \big [\hat \Theta - t(m) \sqrt{v_{11}} \hat \sigma, \,
\hat \Theta + t(m) \sqrt{v_{11}} \hat \sigma \big]
\end{equation*}
where $m=n-p$.
Is this confidence interval admissible? The admissibility of a confidence interval is a much more difficult
concept than the admissibility of a point estimator, since confidence intervals must
satisfy a coverage probability constraint. Also, admissibility of confidence intervals
can be defined in either weak or strong forms (Joshi, 1969, 1982).

Kabaila \& Giri (2009, Section 3) describe a broad class ${\cal D}$ of confidence
intervals that includes $I$. The main result of the present paper, presented in
Section 3, is that $I$ is strongly admissible within the class ${\cal D}$.
An attractive feature of the proof of this
result is that, although lengthy, this proof is quite straightforward and elementary.
Section 2 provides a brief description of this class ${\cal D}$.
For completeness, in Section 4 we describe a strong admissibility result, that follows from the results of Joshi (1969),
for the usual $1-\alpha$ confidence interval for $\theta$ in the somewhat artificial situation
that the error variance $\sigma^2$ is assumed to be known.

\bigskip

\begin{center}
{\large{\bf 2. Description of the class $\boldsymbol{\cal D}$}}
\end{center}

\medskip

Define the parameter $\tau = c^T \beta - t$ where the vector $c$ and the number $t$
are given and $a$ and $c$ are linearly independent.
Let $\hat \tau$ denote $c^T \hat \beta - t$ i.e. the least squares estimator of $\tau$.
Define the matrix $V$ to be the covariance matrix of $(\hat \Theta, \hat \tau)$ divided by $\sigma^2$.
Let $v_{ij}$ denote the $(i,j)$ th element of $V$.
We use the notation $[a \pm b]$ for the interval $[a-b, a+b]$ ($b > 0$).
Define the following confidence interval
for $\theta$
\begin{align}
\label{J(b,s)}
J(b,  s) = \bigg [ \hat \Theta -
\sqrt{v_{11}} \hat \sigma \, b\bigg(\frac{\hat{\tau}}{\hat \sigma \sqrt{v_{22}}}\bigg) \, \pm \,
\sqrt{v_{11}} \hat \sigma \, s\bigg(\frac{|\hat{\tau}|}{\hat \sigma \sqrt{v_{22}}}\bigg)
\bigg ]
\end{align}
where the functions $b$ and $s$ are required to satisfy the following restrictions.
The function $b: \mathbb{R} \rightarrow \mathbb{R}$  is
an odd function and $s: [0, \infty)
\rightarrow (0, \infty)$. Both $b$ and $s$ are bounded. These functions are also
continuous except, possibly, at a finite number of values. Also,
$b(x)=0$ for all $|x| \ge d$ and $s(x)=t(m)$ for all $x \ge d$
where $d$ is a given positive number. Let ${\cal F}(d)$ denote the class
of pairs of functions $(b,s)$ that satisfy these restrictions, for given $d$ ($d>0$).

Define ${\cal D}$ to be the class of all
confidence intervals for $\theta$ of the form \eqref{J(b,s)}, where $c$, $t$, $d$, $b$ and
$s$ satisfy the stated restrictions. Each member of this class is
specified by $(c,t, d, b,s)$.
Apart from the usual $1-\alpha$ confidence interval $I$ for $\theta$,
the class ${\cal D}$ of confidence intervals for $\theta$ includes the following:
\begin{enumerate}

\item
Suppose that we carry out a preliminary hypothesis test of the null hypothesis $\tau = 0$
against the alternative hypothesis $\tau \ne 0$. Also suppose that we construct a confidence interval
for $\theta$ with nominal coverage $1-\alpha$ based on the assumption that the selected model had been
given to us {\it a priori} (as the true model). The resulting confidence interval,
called the naive $1-\alpha$ confidence interval, belongs to the class
${\cal D}$ (Kabaila \& Giri, 2009, Section 2).

\item
Confidence intervals for $\theta$ that are constructed to utilize (in the particular manner described by
Kabaila \& Giri, 2009) uncertain prior information that $\tau=0$.

\end{enumerate}

\noindent Let $K$ denote the usual $1-\alpha$ confidence interval for $\theta$ based on the assumption that $\tau = 0$.
The naive $1-\alpha$ confidence interval, described in (a), may be expressed in the following form:
\begin{equation}
\label{mixture}
h \left(\frac{|\hat \tau|}{\hat \sigma \sqrt{v_{22}}} \right ) I +
\left ( 1 - h \left(\frac{|\hat \tau|}{\hat \sigma \sqrt{v_{22}}} \right ) \right) K
\end{equation}
where $h: [0, \infty) \rightarrow [0,1]$ is the unit step function defined by
$h(x) = 0$ for all $x \in [0,q]$ and $h(x) = 1$ for all $x > q$.
Now suppose that we replace $h$ by a continuous increasing function satisfying $h(0)=0$ and
$h(x) \rightarrow 1$ as $x \rightarrow \infty$ (a similar construction is extensively used
in the context of point estimation by Saleh, 2006). The confidence interval \eqref{mixture}
is also a member of the class ${\cal D}$.

\bigskip

\begin{center}
{\large{\bf 3. Main result}}
\end{center}

\medskip

As noted in Section 2, each member of the class ${\cal D}$ is
specified by $(c,t,d,b,s)$. The following result states that the usual $1-\alpha$ confidence interval
for $\theta$ is strongly admissible within the class ${\cal D}$.

\begin{theorem}
There does not exist $(c,t,d,b,s) \in {\cal D}$ such that the following three conditions hold:
\begin{equation}
\label{exp_len_cond}
\hspace{-2.6cm}(a) \qquad E_{\beta,\sigma^2} \big ( \text{length of } J(b,s) \big ) \le E_{\beta,\sigma^2} \big ( \text{length of } I \big )
\quad \text{for all } (\beta,\sigma^2).
\end{equation}
\begin{equation}
\label{cov_pr_cond}
\hspace{-4.6cm}(b) \qquad P_{\beta,\sigma^2} \big ( \theta \in J(b,s) \big ) \ge P_{\beta,\sigma^2} \big ( \theta \in I \big )
\quad \text{for all } (\beta,\sigma^2).
\end{equation}
$(c)$ \ \ \ \ \ Strict inequality holds in either \eqref{exp_len_cond} or \eqref{cov_pr_cond} for at least one $(\beta,\sigma^2)$.

\end{theorem}

\noindent The proof of this result is presented in Appendix A.

An illustration of this result is provided by Figure 3 of
Kabaila \& Giri (2009). Define $\gamma = \tau/(\sigma \sqrt{v_{22}})$. Also define
\begin{equation*}
e(\gamma;s) = \frac{\text{expected length of $
J(b, s)$}}
{\text{expected length of $I$}}.
\end{equation*}
We call this the scaled expected length of $J(b,s)$. Theorem 1 tells us that for any confidence interval $J(b,s)$,
with minimum coverage probability $1-\alpha$, it cannot be the case that $e(\gamma;s) \le 1$ for all $\gamma$,
with strict inequality for at least one $\gamma$. This fact is illustrated by the bottom panel of Figure 3 of
Kabaila \& Giri (2009).

Define the class $\widetilde {\cal D}$ to be the subset of ${\cal D}$ in which both $b$ and $s$ are continuous functions.
Strong admissibility of the confidence interval $I$ within the class ${\cal D}$ implies weak admissibility of this
confidence interval within the class $\widetilde {\cal D}$, as the following result shows.
Since $(\hat \beta, \hat \sigma^2)$ is a sufficient statistic for
$(\beta, \sigma)$, we reduce the data to $(\hat \beta, \hat \sigma^2)$.

\begin{corollary}
There does not exist $(c,t,d,b,s) \in \widetilde {\cal D}$ such that the following three conditions hold:
\begin{equation}
\label{len_cond}
\hspace{-4.4cm}(a^{\prime}) \qquad \big ( \text{length of } J(b,s) \big ) \le  \big ( \text{length of } I \big )
\quad \text{for all } (\hat \beta,\hat \sigma^2).
\end{equation}
\begin{equation}
\label{cov_pr_cond_repeat}
\hspace{-4.6cm}(b^{\prime}) \qquad P_{\beta,\sigma^2} \big ( \theta \in J(b,s) \big ) \ge P_{\beta,\sigma^2} \big ( \theta \in I \big )
\quad \text{for all } (\beta,\sigma^2).
\end{equation}
$(c^{\prime})$ \ \ \ \ \ Strict inequality holds in either \eqref{len_cond} or \eqref{cov_pr_cond_repeat} for at least one $(\beta,\sigma^2)$.

\end{corollary}

\noindent This corollary is proved in Appendix B.


\newpage

\begin{center}
{\large{\bf 4. Admissibility result for known error variance}}
\end{center}

\medskip

In this section, we suppose that $\sigma^2$ is known. Without loss of generality, we assume that $\sigma^2=1$. As before,
let $\hat \beta$ denote the least squares estimator of $\beta$. Since $\hat \beta$ is a sufficient statistic for $\beta$,
we reduce the data to $\hat \beta$. Assume that the parameter of interest is $\theta = \beta_1 / \sqrt{\text{Var}(\hat \beta_1)}$.
Thus the least squares estimator of $\theta$ is $\hat \Theta = \hat \beta_1 / \sqrt{\text{Var}(\hat \beta_1)}$.
Define
\begin{equation*}
\hat \Delta
= \left[\begin{matrix} \hat \beta_2 - \ell_2 \hat \beta_1 \\ \vdots \\ \hat \beta_p - \ell_p \hat \beta_1 \end{matrix} \right]
\end{equation*}
where $\ell_2, \ldots, \ell_p$ have been chosen such that $\text{Cov}(\hat \beta_j - \ell_j \hat \beta_1, \hat \beta_1)=0$
for $j = 2, \ldots, p$. Now define
\begin{equation*}
\delta
= \left[\begin{matrix} \beta_2 - \ell_2  \beta_1 \\ \vdots \\ \beta_p - \ell_p \beta_1 \end{matrix} \right].
\end{equation*}
Note that $(\hat \Theta, \hat \Delta)$ is obtained by a one-to-one transformation from $\hat \beta$. So,
we reduce the data to $(\hat \Theta, \hat \Delta )$. Note that $\hat \Theta$ and $\hat \Delta $ are independent,
with $\hat \Theta \sim N(\theta, 1)$
and $\hat \Delta $ with a multivariate normal distribution with mean $\delta$ and
known covariance matrix. Define the number $z$ by the requirement that $P(-z \le Z \le z) = 1- \alpha$ for
$Z \sim N(0,1)$. Let $I = \big [ \hat \Theta - z, \hat \Theta + z \big]$. Define
\begin{equation*}
\varphi(\hat \theta, \theta) =
\begin{cases}
1 &\text{if } \theta \in \big [ \hat \theta - z,  \hat \theta + z \big ] \\
0 &\text{otherwise}
\end{cases}
\end{equation*}
This is the probability that $\theta$ is included in the confidence interval $I$, when $\hat \theta$ is the
observed value of $\hat \Theta$. The length of the confidence interval $I$ is
$\int_{-\infty}^{\infty} \varphi(\hat \theta, \theta) \, d\theta = 2z$. Let $p_{\theta}(\cdot)$ denote the probability density function
of $\hat \Theta$ for given $\theta$. The coverage probability of $I$ is
$\int_{-\infty}^{\infty} \varphi(\hat \theta, \theta) \, p_{\theta}(\hat \theta) \, d\hat \theta = 1-\alpha$.

Now let ${\cal C}(\hat \Theta, \hat \Delta)$ denote a confidence set for $\theta$. Define
\begin{equation*}
\varphi_{\delta}(\hat \theta, \theta) = P_{\theta, \delta} \big (\theta \in {\cal C}(\hat \theta, \hat \Delta) \big),
\end{equation*}
where $\hat \theta$ denotes the observed value of $\hat \Theta$. For each given $\delta \in \mathbb{R}^{p-1}$,
the expected Lebesgue measure of ${\cal C}(\hat \Theta, \hat \Delta )$
is $E_{\theta, \delta} \Big (\int_{-\infty}^{\infty} \varphi_{\delta}(\hat \Theta, \theta) \, d\theta \Big )$.
For each given $\delta \in \mathbb{R}^{p-1}$,
the coverage probability of ${\cal C}(\hat \Theta, \hat \Delta )$ is
$\int_{-\infty}^{\infty} \varphi_{\delta}(\hat \theta, \theta) \, p_{\theta}(\hat \theta) \, d\hat \theta$.
Theorem 5.1 of Joshi (1969) implies the following strong admissibility result.
Suppose that $\varphi_{\delta}(\hat \theta, \theta)$ satisfies the following conditions
\begin{enumerate}
\item[(i)]
$E_{\theta, \delta} \Big (\int_{-\infty}^{\infty} \varphi_{\delta}(\hat \theta, \theta)  \, d\theta \Big )
\le E_{\theta, \delta} \Big (\int_{-\infty}^{\infty} \varphi(\hat \theta, \theta) \, d\theta \Big )$ \ for all $\theta \in \mathbb{R}$.

\item[(ii)]
$\int_{-\infty}^{\infty} \varphi_{\delta}(\hat \theta, \theta) \, p_{\theta}(\hat \theta) \, d\hat \theta
\ge \int_{-\infty}^{\infty} \varphi(\hat \theta, \theta) \, p_{\theta}(\hat \theta) \, d\hat \theta$
for all $\theta \in \mathbb{R}$.

\end{enumerate}
Then $\varphi_{\delta}(\hat \theta, \theta) = \varphi(\hat \theta, \theta)$ for almost all
$(\hat \theta, \theta) \in \mathbb{R}^2$. This result is true
for each $\delta \in \mathbb{R}^{p-1}$.
Using standard arguemnts, this entails that
$I \setminus {\cal C}(\hat{\Theta}, \hat{\Delta})$ and
${\cal C}(\hat{\Theta}, \hat{\Delta})\setminus I$ are Lebesgue-null sets,
for (Lebesgue-) almost all values of $(\hat{\Theta}, \hat{\Delta})$.

\bigskip

\begin{center}
{\large{\bf Appendix A: Proof of Theorem 1}}
\end{center}

\medskip

Suppose that $c$ is a given vector (such that $c$ and $a$ are linearly
independent), $t$ is a given number and $d$ is a given positive
number. The proof of Theorem 1 now proceeds as follows.
We present a few definitions and a lemma.
We then apply this lemma to prove this theorem.

Define $W = \hat \sigma/\sigma$.
Note that $W$ has the same distribution as $\sqrt{Q/m}$ where $Q
\sim \chi^2_m$.
Let $f_W$ denote the probability density function of $W$.
Also let $\phi$ denote the $N(0,1)$ probability density function.
Now define
\begin{equation*}
R_1(b,s;\gamma) = \frac{\text{expected length of $
J(b, s)$}}
{\text{expected length of $I$}} - 1.
\end{equation*}
It follows from (7) of Kabaila \& Giri (2009) that
\begin{equation}
\label{R1}
R_1(b,s;\gamma) = \frac{1} {t(m) \, E(W)}
 \int^{\infty}_0 \int^{d}_{- d} \left (s(|x|) - t(m) \right )
\phi(w x -\gamma) \, dx \,  w^2 \,  f_W(w)  \, dw.
\end{equation}
Thus, for each $(b,s) \in {\cal F}(d)$,
$R_1(b,s;\gamma)$ is a continuous function of $\gamma$.

Also define
$R_2(b,s;\gamma) = P \big(\theta \notin J(b,s) \big) - \alpha$.
We make the following definitions, also used by Kabaila \& Giri (2009).
Define $\rho = v_{12}/\sqrt{v_{11} v_{22}}$ and $\Psi(x, y; \mu, v) = P(x \le Z \le y)$,
for $Z \sim N(\mu,v)$.
Now define the functions
\begin{align*}
k^{\dag}(h,w, \gamma, \rho) &= \Psi \big( -t(m) w, t(m) w;
 \rho(h-\gamma), 1-\rho^2 \big ) \\
k(h,w,\gamma, \rho) &= \Psi \big(b(h/w) w - s(|h|/w) w, b(h/w) w + s(|h|/w) w; \rho(h-\gamma),1-\rho^2 \big ).
\end{align*}
It follows from (6) of Kabaila \& Giri (2009), that
\begin{equation}
\label{R2}
R_2(b,s; \gamma) = -\int_0^{\infty} \int_{-d}^{d} \big( k(wx,w, \gamma, \rho) - k^{\dag}(wx,w, \gamma, \rho) \big)
\, \phi(wx-\gamma)\,  dx \, w \, f_W(w) \, dw .
\end{equation}
Thus, for each $(b,s) \in {\cal F}(d)$,
$R_2(b,s;\gamma)$ is a continuous function of $\gamma$.

Now $E(W^2)=1$ and so
\begin{equation*}
\int_0^{\infty} w^2 \, f_W(w) \, dw = 1.
\end{equation*}
It follows from \eqref{R1} that
\begin{equation}
\label{int_R1}
\int_{-\infty}^{\infty} R_1(b,s;\gamma) \, d\gamma = \frac{2} {t(m) \, E(W)}
 \int^{d}_0 \big (s(x) - t(m) \big ) \, dx.
\end{equation}
Thus $\int_{-\infty}^{\infty} R_1(b,s;\gamma)\, d\gamma$ exists for all $(b,s) \in {\cal F}(d)$.

Since $k(wx,w, \gamma, \rho)$ and $k^{\dag}(wx,w, \gamma, \rho)$ are probabilities,
\begin{equation*}
|R_2(b,s;\gamma)| \le \int_0^{\infty} \int_{-d}^d \phi(wx-\gamma) dx \, w f_W(w) \, dw,
\end{equation*}
so that
\begin{equation*}
\int_{-\infty}^{\infty} |R_2(b,s;\gamma)| \, d \gamma \le 2 d \int_0^{\infty}  w f_W(w) \, dw = 2 d E(W) < \infty.
\end{equation*}
Thus $\int_{-\infty}^{\infty} R_2(b,s;\gamma)\, d\gamma$ exists for all
 $(b,s) \in {\cal F}(d)$.

Thus, we may define
\begin{equation*}
g(b,s;\lambda) = \lambda \int_{-\infty}^{\infty} R_1(b,s;\gamma)\, d\gamma + (1-\lambda) \int_{-\infty}^{\infty} R_2(b,s;\gamma)\, d\gamma,
\end{equation*}
for each $(b,s) \in {\cal F}(d)$, where $0 < \lambda < 1$.
Kempthorne (1983, 1987, 1988) presents results on what he
calls compromise decision theory. Initially, these results were applied only to the solution
of some problems of point estimation. Kabaila \& Tuck (2008) develop new results in compromise decision theory and apply these to
a problem of interval estimation. The following lemma, which will be used in the proof
of Theorem 1, is in the style of these compromise decision theory results.

\begin{lemma}
Suppose that $c$ is a given vector (such that $c$ and $a$ are linearly
independent), $t$ is a given number and $d$ is a given positive
number.
Also suppose that $\lambda$ is given and that $(b^*,s^*)$ minimizes $g(b,s;\lambda)$ with respect to
$(b,s) \in {\cal F}(d)$. Then there does not exist $(b,s) \in {\cal F}(d)$ such that
\begin{enumerate}
\item
$R_1(b,s;\gamma) \le R_1(b^*,s^*;\gamma)$ for all $\gamma$.

\item
$R_2(b,s;\gamma) \le R_2(b^*,s^*;\gamma)$ for all $\gamma$.

\item
Strict inequality holds in either (a) or (b) for at least one $\gamma$.
\end{enumerate}
\end{lemma}

\begin{proof}
Suppose that $c$ is a given vector (such that $c$ and $a$ are linearly
independent), $t$ is a given number and $d$ is a given positive
number.
The proof is by contradiction. Suppose that there exist $(b,s) \in {\cal F}(d)$ such that
$(a)$, $(b)$ and $(c)$ hold. Now,
\begin{align*}
g(b^*, s^*;\lambda) - g(b, s;\lambda) &=
\lambda \int_{-\infty}^{\infty} \big (R_1(b^*, s^*;\gamma) - R_1(b, s;\gamma) \big )\, d\gamma \\
&\phantom{123}+
(1-\lambda) \int_{-\infty}^{\infty} \big (R_2(b^*, s^*;\gamma) - R_2(b,s;\gamma) \big )\, d\gamma
\end{align*}
By hypothesis, one of the following 2 cases holds.

\medskip

\noindent \underbar{Case 1} \ $(a)$ and $(b)$ hold and
$R_1(b^*, s^*;\gamma) - R_1(b, s;\gamma) > 0$ for at least one $\gamma$.
Since $R_1(b^*, s^*;\gamma) - R_1(b, s;\gamma)$ is a continuous function
of $\gamma$,
\begin{equation*}
\int_{-\infty}^{\infty} \big (R_1(b^*, s^*;\gamma) - R_1(b, s;\gamma) \big )\, d\gamma > 0.
\end{equation*}
Thus $g(b^*, s^*;\lambda) > g(b, s;\lambda)$ and we have established a
contradiction.

\medskip

\noindent \underbar{Case 2} \ $(a)$ and $(b)$ hold and
$R_2(b^*, s^*;\gamma) - R_2(b, s;\gamma) > 0$ for at least one $\gamma$.
Since $R_2(b^*, s^*;\gamma) - R_2(b, s;\gamma)$ is a continuous function
of $\gamma$,
\begin{equation*}
\int_{-\infty}^{\infty} \big (R_2(b^*, s^*;\gamma) - R_2(b, s;\gamma) \big )\, d\gamma > 0.
\end{equation*}
Thus $g(b^*, s^*;\lambda) > g(b, s;\lambda)$ and we have established a
contradiction.

\medskip

\noindent Lemma 1 follows from the fact that this argument holds for every given vector $c$ (such that $c$ and $a$ are linearly
independent), every given number $t$ and every given positive number $d$.

\end{proof}

\medskip

 We will first find the $(b^*,s^*)$ that minimizes
$g(b,s;\lambda)$ with respect to $(b,s) \in {\cal F}(d)$, for given $\lambda$.
We will then choose $\lambda$ such that $J(b^*, s^*) = I$, the usual $1-\alpha$ confidence interval for
$\theta$.
Theorem 1 is then a consequence of Lemma 1.

By changing the variable of integration in the inner integral
in \eqref{R2}, it can be shown that $R_2(b,s; \gamma)$
is equal to
\begin{align*}
-\int_0^{\infty} \int_0^{d} \Big ( &\big( k(wx,w, \gamma, \rho) - k^{\dag}(wx,w, \gamma, \rho) \big)
\, \phi(wx-\gamma) + \\
&\big( k(-wx,w, \gamma, \rho) - k^{\dag}(-wx,w, \gamma, \rho) \big)
\, \phi(wx+\gamma) \Big) \,  dx \, w \, f_W(w) \, dw
\end{align*}
Using this expression and the restriction that $b$ is an odd function, we find that
$\int_{-\infty}^{\infty} R_2(b,s; \gamma) \, d \gamma$ is equal to
\begin{align*}
-\int_0^d  \int_0^{\infty} \int_{-\infty}^{\infty} \Big (
&\Psi \big(b(x)w-s(x)w, b(x)w + s(x)w; \rho y, 1-\rho^2 \big) \\
&-\Psi \big(-t(m)w, t(m)w; \rho y, 1-\rho^2 \big) \\
&+\Psi \big(-b(x)w-s(x)w, -b(x)w + s(x)w; -\rho y, 1-\rho^2 \big) \\
&-\Psi \big(-t(m)w, t(m)w; -\rho y, 1-\rho^2 \big) \Big)
\, \phi(y) \, dy \,w \, f_W(w) \, dw \, dx .
\end{align*}
Hence, to within an additive constant that does
not depend on $(b,s)$, $\int_{-\infty}^{\infty} R_2(b,s; \gamma) \, d \gamma$ is equal to
\begin{align*}
-\int_0^d  \int_0^{\infty} &\int_{-\infty}^{\infty} \Big (
\Psi \big(b(x)w-s(x)w, b(x)w + s(x)w; \rho y, 1-\rho^2 \big) \\
&+\Psi \big(-b(x)w-s(x)w, -b(x)w + s(x)w; -\rho y, 1-\rho^2 \big)
\Big)
\, \phi(y) \, dy \,w \, f_W(w) \, dw \, dx .
\end{align*}
Thus, to within an additive constant that does not depend on $(b,s)$,
\begin{equation*}
g(b,s;\lambda) = \int_0^d q(b,s;x) \, dx,
\end{equation*}
where $q(b,s;x)$ is equal to
\begin{align*}
&\frac{2 \lambda} {t(m) \, E(W)}
 s(x) \\
 &-(1-\lambda) \int_0^{\infty} \int_{-\infty}^{\infty} \big (\Psi(b(x)w-s(x)w, b(x)w + s(x)w; \rho y, 1-\rho^2) \\
&\phantom{1234567}+\Psi(-b(x)w-s(x)w, -b(x)w + s(x)w; -\rho y, 1-\rho^2)
\big)
\, \phi(y) \, dy \,w \, f_W(w) \, dw.
\end{align*}
Note that $x$ enters into the expression
for $q(b,s;x)$ only through $b(x)$ and $s(x)$.
To minimize $g(b,s;\lambda)$ with respect to $(b,s) \in {\cal F}(d)$, it is therefore sufficient to minimize
$q(b,s;x)$ with respect to $(b(x),s(x))$ for each $x \in [0,d]$.
The situation here is similar to the computation of Bayes rules, see e.g. Casella \& Berger (2002, pp. 352--353).
Therefore, to minimize $g(b,s;\lambda)$ with respect to
$(b,s) \in {\cal F}(d)$, we simply minimize
\begin{align*}
\tilde q(b,s) = &\frac{2 \lambda} {t(m) \, E(W)}
 s \\
 &-(1-\lambda) \int_0^{\infty} \int_{-\infty}^{\infty} \big (\Psi(bw-sw, bw + sw; \rho y, 1-\rho^2) \\
&\phantom{1234567890}+\Psi(-bw-sw, -bw + sw; -\rho y, 1-\rho^2)
\big)
\, \phi(y) \, dy \,w \, f_W(w) \, dw
\end{align*}
with respect to $(b,s) \in \mathbb{R} \times (0,\infty)$, to obtain $(b^{\prime}, s^{\prime})$ and then
set $b(x) = b^{\prime}$ and $s(x) = s^{\prime}$ for all $x \in [0,d]$.

Let the random variables $A$ and $B$ have the following distribution
\begin{equation*}
\left[\begin{matrix} A\\ B \end{matrix}
\right] \sim N \left ( \left[\begin{matrix} 0 \\ 0 \end{matrix}
\right], \left[\begin{matrix} 1 \quad \rho\\ \rho \quad 1 \end{matrix}
\right] \right ).
\end{equation*}
Note that the distribution of $A$, conditional on $B=y$, is $N(\rho y, 1-\rho^2)$.
Thus
\begin{equation*}
\Psi(bw - sw, bw + sw; \rho y, 1-\rho^2) = P \big(bw - sw \le A \le bw + sw \, \big| \, B=y \big)
\end{equation*}
Hence
\begin{align}
\label{first_to_maximize}
&\int_0^{\infty} \int_{-\infty}^{\infty} \Psi(bw-sw, bw + sw; \rho y, 1-\rho^2)
\, \phi(y) \, dy \,w \, f_W(w) \, dw \notag \\
&=\int_0^{\infty} P (bw - sw \le A \le bw + sw ) \,w \, f_W(w) \, dw .
\end{align}
Let $\Phi$ denote the $N(0,1)$ cumulative distribution function.
For every fixed $w>0$ and $s>0$,
\begin{equation*}
P (bw - sw \le A \le bw + sw ) = \Phi(bw + sw) - \Phi(bw - sw)
\end{equation*}
is maximized by setting $b=0$. Thus, for each fixed $s>0$, \eqref{first_to_maximize}
is maximized with respect to $b \in \mathbb{R}$ by setting $b=0$.

Now let the random variables $\tilde A$ and $\tilde B$ have the following distribution
\begin{equation*}
\left[\begin{matrix} \tilde A\\ \tilde B \end{matrix}
\right] \sim N \left ( \left[\begin{matrix} 0 \\ 0 \end{matrix}
\right], \left[\begin{matrix} 1 \quad -\rho\\ -\rho \quad 1 \end{matrix}
\right] \right ).
\end{equation*}
Note that the distribution of $\tilde A$, conditional on $\tilde B=y$, is $N(-\rho y, 1-\rho^2)$.
Thus
\begin{equation*}
\Psi(-bw - sw, -bw + sw; -\rho y, 1-\rho^2) = P \big(-bw - sw \le \tilde A \le -bw + sw \, \big| \, \tilde B=y \big)
\end{equation*}
Hence
\begin{align}
\label{second_to_maximize}
&\int_0^{\infty} \int_{-\infty}^{\infty} \Psi(-bw-sw, -bw + sw; -\rho y, 1-\rho^2)
\, \phi(y) \, dy \,w \, f_W(w) \, dw \notag \\
&=\int_0^{\infty} P (-bw - sw \le \tilde A \le -bw + sw ) \,w \, f_W(w) \, dw .
\end{align}
For every fixed $w>0$ and $s>0$,
\begin{equation*}
P \big(-bw - sw \le \tilde A \le -bw + sw \big) = \Phi(-bw + sw) - \Phi(-bw - sw)
\end{equation*}
is maximized by setting $b=0$. Thus, for each fixed $s>0$, \eqref{second_to_maximize}
is maximized with respect to $b \in \mathbb{R}$ by setting $b=0$.

Therefore, $\tilde q(b,s)$ is, for each fixed $s > 0$, minimized with respect to $b$ by
setting $b=0$. Thus $b^{\prime}=0$ and so $b^*(x) = 0$ for all $x \in \mathbb{R}$.
Hence, to find $s^{\prime}$ we need to minimize
\begin{equation*}
\frac{\lambda}{t(m) E(W)} s - (1-\lambda) \int_0^{\infty} \big ( 2 \Phi(sw) - 1 \big) \, w f_W(w) \, dw
\end{equation*}
with respect to $s > 0$. Therefore, to find $s^{\prime}$ we may minimize
\begin{equation*}
r(s) = \ell(\lambda) \, s - 2 \int_0^{\infty} \Phi(sw) \, w f_W(w) \, dw
\end{equation*}
with respect to $s > 0$, where
\begin{equation*}
\ell(\lambda) = \frac{\lambda}{(1-\lambda) t(m) E(W)}.
\end{equation*}
Note that $\ell(\lambda)$ is an increasing function of $\lambda$,
such that $\ell(\lambda) \downarrow 0$ as $\lambda \downarrow 0$ and
$\ell(\lambda) \uparrow \infty$ as $\lambda \uparrow 1$.
Choose $\lambda = \lambda^*$, where
\begin{equation*}
\ell(\lambda^*) = 2 \int_0^{\infty} \phi \big( t(m) w \big) \, w^2 \, f_W(w) \, dw.
\end{equation*}
Note that $0 < \ell(\lambda^*) < \sqrt{2 / \pi}$. Now
\begin{equation*}
\frac{d r(s)}{ds} = \ell(\lambda^*)  - 2 \int_0^{\infty} \phi(sw) \, w^2 f_W(w) \, dw .
\end{equation*}
Since $\int_0^{\infty} \phi(sw) \, w^2 f_W(w) \, dw$ is a decreasing function of $s > 0$,
$dr(s)/ds$ is an increasing function of $s > 0$.
Also, for $s=0$, $\int_0^{\infty} \phi(sw) \, w^2 f_W(w) \, dw = 1/\sqrt{2 \pi}$.
Thus, to minimize
$r(s)$ with respect to $s > 0$, we need to solve
\begin{equation*}
\ell(\lambda^*) - 2 \int_0^{\infty} \phi(sw) \, w^2 \, f_W(w) \, dw = 0
\end{equation*}
for $s > 0$. Obviously, this solution in $s = t(m)$.
Thus $s^*(x) = t(m)$ for all $x \ge 0$. In other words, $J(b^*, s^*) = I$.
By Lemma 1, there does not exist $(b,s) \in {\cal F}(d)$ such that
\begin{equation}
\label{exp_len_cond_repeat}
\hspace{-2.6cm}(a) \qquad E_{\beta,\sigma^2} \big ( \text{length of } J(b,s) \big ) \le E_{\beta,\sigma^2} \big ( \text{length of } I \big )
\quad \text{for all } (\beta,\sigma^2).
\end{equation}
\begin{equation}
\label{cov_pr_cond_repeat}
\hspace{-4.6cm}(b) \qquad P_{\beta,\sigma^2} \big ( \theta \in J(b,s) \big ) \ge P_{\beta,\sigma^2} \big ( \theta \in I \big )
\quad \text{for all } (\beta,\sigma^2).
\end{equation}
\ $(c)$ \ \ \ \ Strict inequality holds in either \eqref{exp_len_cond_repeat} or \eqref{cov_pr_cond_repeat} for at least one $(\beta,\sigma^2)$.

\medskip

\noindent Theorem 1 follows from the fact that this argument holds for every given vector $c$ (such that $c$ and $a$ are linearly
independent), every given number $t$ and every given positive number $d$.

\bigskip

\begin{center}
{\large{\bf Appendix B: Proof of Corollary 1}}
\end{center}

\medskip

The proof of Corollary 1 is by contradiction. Suppose that $c$ is a given vector (such that $c$ and $a$ are linearly
independent), $t$ is a given number and $d$ is a given positive
number. Also suppose that there exists $(b,s) \in {\cal F}(d)$
such that both $b$ and $s$ are continuous and $(a^{\prime})$, $(b^{\prime})$ and $(c^{\prime})$, in the statement of Corollary 1, hold.
Now $(a^{\prime})$ implies that
\begin{equation*}
E_{\beta,\sigma^2} \big ( \text{length of } J(b,s) \big ) \le E_{\beta,\sigma^2} \big ( \text{length of } I \big )
\quad \text{for all } (\beta,\sigma^2),
\end{equation*}
so that $(a)$ holds.
By hypothesis, one of the following two cases holds.

\medskip

\noindent \underbar{Case 1} \ \  $\big ( \text{length of } J(b,s) \big ) <  \big ( \text{length of } I \big )
\quad \text{for at least one } (\hat \beta,\hat \sigma^2)$. Now
\begin{equation*}
\big ( \text{length of } J(b,s) \big ) =
2 \sqrt{v_{11}} \hat \sigma \, s \left ( \frac{|\hat \tau|}{\hat \sigma \sqrt{v_{22}}} \right ),
\end{equation*}
which is a continuous function of $(\hat \beta,\hat \sigma^2)$.
Hence $\big ( \text{length of } I \big ) - ( \text{length of } J(b,s) \big )$ is
a continuous function of $(\hat \beta,\hat \sigma^2)$.
Thus
\begin{equation*}
E_{\beta,\sigma^2} \big ( \text{length of } J(b,s) \big ) < E_{\beta,\sigma^2} \big ( \text{length of } I \big )
\quad \text{for at least one } (\beta,\sigma^2).
\end{equation*}
Thus there exists $(b,s) \in {\cal F}(d)$ such that $(a)$, $(b)$ and $(c)$, in the statement of Theorem 1, hold.
We have established a contradiction.

\medskip

\noindent \underbar{Case 2} \ There is strict inequality in $(b^{\prime})$ for at least one $(\beta,\sigma^2)$.
Thus there exists $(b,s) \in {\cal F}(d)$ such that $(a)$, $(b)$ and $(c)$, in the statement of Theorem 1, hold.
We have established a contradiction.

\medskip

\noindent Corollary 1 follows from the fact that this argument holds for every given vector $c$ (such that $c$ and $a$ are linearly
independent), every given number $t$ and every given positive number $d$.

\begin{center}
\nin{\sl References}
\end{center}

\smallskip

\rf {\sc CASELLA, G. \& BERGER, R.L.} (2002). {\sl Statistical Inference}, 2nd ed.. Pacific Grove, CA: Duxbury.

\rf {\sc JOSHI, V.M.} (1969). Admissibility of the usual confidence sets for the mean of a univariate
or bivariate normal population. {\sl Annals of Mathematical Statistics}, {\bf 40}, 1042--1067.

\smallskip

\rf {\sc JOSHI, V.M.} (1982). Admissibility. On pp.25--29 of Vol. 1 of {\sl Encyclopedia of Statistical Sciences},
 editors-in-chief, Samuel Kotz, Norman L. Johnson ; associate editor, Campbell B. Read.
New York: John Wiley.

\smallskip

\rf {\sc KABAILA, P. \& GIRI, K.} (2009). Confidence intervals in regression utilizing prior
information. {\sl Journal of Statistical Planning and Inference}, {\bf 139}, 3419--3429.

\smallskip

\rf {\sc KABAILA, P. \& TUCK, J.} (2008). Confidence intervals utilizing prior
information in the Behrens-Fisher problem. {\sl Australian \& New Zealand Journal of Statistics}
{\bf 50}, 309--328.

\smallskip

\rf {\sc KEMPTHORNE, P.J.} (1983). Minimax-Bayes compromise estimators. In
{\sl 1983 Business and Economic Statistics Proceedings of the
American Statistical Association}, Washington DC, pp.568--573.

\smallskip

\rf {\sc KEMPTHORNE, P.J.} (1987).  Numerical specification of
discrete least favourable prior distributions.  SIAM
{\sl Journal on Scientific and Statistical Computing} {\bf 8}, 171--184.

\smallskip

\rf {\sc KEMPTHORNE, P.J.} (1988). Controlling risks under different loss
functions: the compromise decision problem. {\sl Ann.
Statist.} {\bf 16}, 1594--1608.

\smallskip

\rf {\sc SALEH, A.K.Md.E.} (2006) Theory of Preliminary Test and Stein-Type Estimation
with Applications. Hoboken, NJ: John Wiley.

\end{document}